\documentclass[reqno]{amsart}
\usepackage{amssymb}
\usepackage{graphicx}

\usepackage[usenames, dvipsnames]{color}
\usepackage[colorlinks,linkcolor=blue,citecolor=red]{hyperref}
\usepackage{verbatim}
\usepackage{mathrsfs}
\usepackage{esint}
\usepackage{empheq}

\numberwithin{equation}{section}

\newtheorem{theorem}{Theorem}[section]

\newtheorem{lemma}[theorem]{Lemma}
\newtheorem{prop}[theorem]{Proposition}

\theoremstyle{definition}
\newtheorem{remark}[theorem]{Remark}

\theoremstyle{definition}

\theoremstyle{definition}

\makeatletter
\def\dashint{\operatorname%
{\,\,\text{\bf-}\kern-.98em\DOTSI\intop\ilimits@\!\!}}
\makeatother

\newcommand{\RN}[1]{%
  \textup{\uppercase\expandafter{\romannumeral#1}}%
}

\newcounter{marnote}

\begin{document}
\title[Entire solutions of Hessian equations]
{Liouville property and existence of entire solutions of Hessian equations}

\author[C. Wang]{Cong Wang}
\address[C. Wang]{School of Mathematical Sciences, Beijing Normal University, Laboratory of Mathematics and Complex Systems, Ministry of Education, Beijing 100875, China.}
\email{cwang@mail.bnu.edu.cn}

\author[J.G. Bao]{Jiguang Bao}
\address[J.G. Bao]{School of Mathematical Sciences, Beijing Normal University, Laboratory of Mathematics and Complex Systems, Ministry of Education, Beijing 100875, China.}
\email{jgbao@bnu.edu.cn}
\thanks{J.G. Bao was supported by the National Natural Science Foundation of China (No. 11871102)}






\begin{abstract}
In this paper, we establish the existence and uniqueness theorem for entire solutions of Hessian equations with prescribed asymptotic behavior at infinity. This extends the previous results on Monge-Amp\`{e}re equations.  Our approach also makes the prescribed asymptotic order optimal within the range preset in exterior Dirichlet problems. In addition, we show a Liouville type result for $k$-convex solutions. This partly removes the $(k+1)$- or $n$-convexity restriction imposed in existing work.
\end{abstract}

\maketitle

\section{Introduction}

In this paper, we study the Liouville property of Hessian equations
\begin{equation}\label{eq:sigma_k=1}
\sigma_k(\lambda(D^2u))=1
\end{equation}
and the existence of solutions of
\begin{equation}\label{eq:k-Hessian=f}
\sigma_k(\lambda(D^2u))=f
\end{equation}
in $\mathbb{R}^n$ with $f$ a perturbation of $1$ near infinity. Here $\lambda(D^2u)$ denotes the eigenvalue vector $\lambda:=(\lambda_1,\cdots,\lambda_n)$ of the Hessian matrix $D^2u$,
$$\sigma_k(\lambda)=\sum_{1\leq i_1<\cdots<i_k\leq n}\lambda_{i_1}\cdots\lambda_{i_k}$$
is the $k$-th elementary symmetric function, $k=1,\cdots,n$. Note that for $k=1$, \eqref{eq:k-Hessian=f} corresponds to Possion's equation $\Delta u=f$, which is a linear elliptic equation. For $2\leq k\leq n$, \eqref{eq:k-Hessian=f} is an important class of fully nonlinear elliptic equations. Especially, for $k=n$, it corresponds to the famous Monge-Amp\`{e}re equation $\text{det}(D^2u)=f$.

For the case that the right hand side $f\equiv1$, continuing attention is paid to Liouville properties of \eqref{eq:sigma_k=1}. For $k=1$, the classical Liouville theorem for harmonic functions shows that any convex entire classical solution of \eqref{eq:sigma_k=1} must be a quadratic polynomial. For $k=n$, the celebrated theorem of J\"{o}rgens \cite{J 1954}, Calabi \cite{C 1958} and Pogorelov \cite{P 1972} for the Monge-Amp\`{e}re equation states that any convex entire classical solution of \eqref{eq:sigma_k=1} must be a quadratic polynomial. Different proofs of J\"{o}rgens-Calabi-Pogorelov theorem were given by Cheng-Yau \cite{CT 1986}, Caffarelli \cite{C 1995} and Jost-Xin \cite{JX 2001}. For $k=2$, Chang-Yuan \cite{CY 2010} proved that if $u$ is an entire solution of \eqref{eq:sigma_k=1} with
$$D^2u\geq\left(\delta-\sqrt{\frac{2}{n(n-1)}}\right)I$$
for some $\delta>0$, then $u$ is a quadratic polynomial. Recently, Shankar-Yuan \cite{SY 2022} improved this result to general semiconvex solutions, i.e.
$$D^2u\geq -KI$$
for a large $K>0$. For general $k$, assuming a lower quadratic growth condition, Bao-Chen-Guan-Ji \cite{BCGJ 2003} demonstrated that any strictly convex solution of \eqref{eq:sigma_k=1} is a quadratic polynomial. Afterwards, Li-Ren-Wang \cite{LRW 2016} relaxed the strict convexity restriction in \cite{BCGJ 2003} to $(k+1)$-convexity. However, to the best of our knowledge, there is little known about the Liouville type result for $k$-convex solutions of \eqref{eq:sigma_k=1}. We refer to \cite{B 2004}, where Bao established such a result by introducing a integral growth condition for $D^2u$.

For the general right hand side, Caffarelli-Li \cite{CL 2003} generalized the J\"{o}rgens-Calabi-Pogorelov theorem. They considered
\begin{equation}\label{eq:det=f}
\text{det}(D^2u)=f\quad\text{in}\ \mathbb{R}^n,
\end{equation}
where $f\in C^0(\mathbb{R}^n)$ satisfies
$$\inf_{\mathbb{R}^n}f>0\quad\text{and}\quad\mathrm{supp}(f-1)\ \text{is~bounded}.$$
For $n\geq3$, they proved that any convex viscosity solution of \eqref{eq:det=f} approaches a quadratic polynomial at infinity with
\begin{equation}\label{det-asy}
\limsup_{|x|\rightarrow\infty}|x|^{n-2}\left|u(x)-\left(\frac{1}{2}x^TAx+b\cdot x+c\right)\right|<\infty,
\end{equation}
where $A$ is a symmetric positive definite matrix with $\text{det}(A)=1$, $b\in\mathbb{R}^n$ and $c\in\mathbb{R}$. With such prescribed asymptotic behavior near infinity, they also established an existence and uniqueness theorem for solutions of \eqref{eq:det=f}. Extending \cite{CL 2003}, Bao-Li-Zhang \cite{BLZ 2015} considered \eqref{eq:det=f} where $f\in C^0(\mathbb{R}^n)$ satisfies
\begin{subequations}
\begin{align}
&\inf_{\mathbb{R}^n}f>0\ \text{and}\ f\in C^3(\mathbb{R}^n\setminus D),\label{C1}\\
&\exists\ \beta>2\ \text{such~that}\ \limsup_{|x|\rightarrow\infty}|x|^{\beta+m}|D^m(f(x)-1)|<\infty,\ m=0,1,2,3,\label{C2}
\end{align}
\end{subequations}
where $D\subset\mathbb{R}^n$ is a bounded open set in \eqref{C1}.  In the same spirit as \cite{CL 2003}, they derived asymptotic behavior of solutions near infinity,
\begin{equation}\label{eq:asymptotic}
\begin{cases}
\limsup_{|x|\rightarrow\infty}|x|^{\min\{\beta,n\}-2}\left|u(x)-\left(\frac{1}{2}x^TAx+b\cdot x+c\right)\right|<\infty,&\text{if}\ \beta\neq n,\\
\limsup_{|x|\rightarrow\infty}|x|^{n-2}(\ln|x|)^{-1}\left|u(x)-\left(\frac{1}{2}x^TAx+b\cdot x+c\right)\right|<\infty,&\text{if}\ \beta=n.
\end{cases}
\end{equation}
Also, with this prescribed asymptotic behavior, they proved the existence and uniqueness theorem for solutions of \eqref{eq:det=f} (the case $\beta=n$ was missed in the original paper \cite{BLZ 2015}). Note that the above asymptotic results are extensions for J\"{o}rgens-Calabi-Pogorelov theorem. In dimension two, the asymptotic results were studied by Ferrer–Mart\'{\i}nez–Mil\'{a}n \cite{FMM 2000} and Bao-Li-Zhang \cite{BLZ 2015}. See also Bao-Xiong-Zhou \cite{BXZ 2019} for the existence of solutions of \eqref{eq:det=f} with prescribed asymptotic behavior.

We would like to mention a further extension of J\"{o}rgens-Calabi-Pogorelov theorem. In another paper \cite{CL 2004}, Caffarelli-Li proved that any classical convex solution of \eqref{eq:det=f} with a periodic positive $f$ must be sum of a quadratic polynomial and a periodic function. See also the work of Teixeira-Zhang \cite{TZ 2016} for a perturbation of periodic positive functions $f$.

In this paper, we focus our attention on Hessian equations. We obtain a Liouville property for $k$-convex entire solutions of \eqref{eq:sigma_k=1}. Under an additional upper quadratic growth condition, this removes the assumption of $(k+1)$- or $n$-convexity in \cite{LRW 2016} and \cite{BCGJ 2003}, respectively. Furthermore, we prove the existence and uniqueness for entire solutions of \eqref{eq:k-Hessian=f} with prescribed asymptotic behavior \eqref{eq:asymptotic} at infinity. This generalizes the previous corresponding results (see \cite{CL 2003} and \cite{BLZ 2015}) on Monge-Amp\`{e}re equations to Hessian equations.

Recall that (see \cite{CNS 1985}) for an open set $\Omega\subset\mathbb{R}^n$, we say that a function $u\in C^2(\Omega)$ is $k$-convex, if $\lambda(D^2u(x))\in\overline{\Gamma}_k$ for every $x\in\Omega$, where $\Gamma_k$ is an open convex symmetric cone in $\mathbb{R}^n$ with its vertex at the origin, given by
$$\Gamma_k=\{(\lambda_1,\lambda_2,\cdots,\lambda_n)\in\mathbb{R}^n: \sigma_j(\lambda)>0,\  \forall\ j=1,\cdots,k\}.$$

Our first main result is

\begin{theorem}\label{thm:main}
Let $u\in C^{3,1}(\mathbb{R}^n)$ be a $k$-convex solution of \eqref{eq:sigma_k=1} in $\mathbb{R}^n$. If there exist positive constants $A_1$, $A_2$, $B$ and $R_0$ such that
\begin{equation}\label{eq:quad-growth2}
A_1|x|^2\leq u(x)\leq A_2|x|^2+B\quad\text{for}\ |x|\geq R_0.
\end{equation}
Then $u$ is a quadratic polynomial.
\end{theorem}

We remark that it is of interest to investigate the Liouville property for $k$-convex solutions due to the natural ellipticity class for Hessian equations.
Benefiting from the fact that level sets of convex functions are between two balls after some affine transformation (see \cite{de G 1975}), level set method is an effective way to study Liouville properties of Monge-Amp\`{e}re equations (see e.g., \cite{BLZ 2015}, \cite{CL 2003} and \cite{TW 2008}). However, the level sets of $k$-convex functions could even be unbounded. It is at this point where main changes have to be made to control the level sets of $k$-convex solutions. For this purpose, based on the lower quadratic growth condition in \cite{BCGJ 2003} and \cite{LRW 2016}, we further introduce the upper quadratic growth condition in Theorem \ref{thm:main}.

Denote
$$\mathcal{A}_k=\{A: A~\text{is~real}~n\times n~\text{symmetric~positive~definite~matrix~and}~\sigma_k(\lambda(A))=1\}.$$
The second main result of the paper is

\begin{theorem}\label{thm:main-thm}
Let $n\geq3$ and $f\in C^{2,\gamma}_{\mathrm{loc}}(\mathbb{R}^n)$ satisfy \eqref{C1}-\eqref{C2} for some $0<\gamma<1$. Then for any $A\in\mathcal{A}_k$, $b\in\mathbb{R}^n$ and $c\in\mathbb{R}$,
there exists a unique $k$-convex solution $u\in C^{4,\gamma}_{\mathrm{loc}}(\mathbb{R}^n)$ of \eqref{eq:k-Hessian=f} in $\mathbb{R}^n$ satisfying \eqref{eq:asymptotic}.
\end{theorem}

\begin{remark}
For the special cases that $k=n=2$ and that $k=n\geq3$, the corresponding results have been proved in \cite{BXZ 2019} and \cite{BLZ 2015}, respectively.
\end{remark}

Now let us comment the order in \eqref{eq:asymptotic} that the solution of \eqref{eq:k-Hessian=f} approaches the quadratic polynomial near infinity. For the exterior Dirichlet problems for \eqref{eq:k-Hessian=f}, the prescribed order is approximate to that in \eqref{eq:asymptotic} with $n$ replaced by $\theta n$, where $\theta\in\left[\frac{k}{n},1\right]$ (see, e.g., \cite{BLL 2014} and \cite{CB 2017}). Therefore, we give a finer order that may be attained for solutions of \eqref{eq:k-Hessian=f} defined in exterior domains. A key step in proving Theorem \ref{thm:main-thm} is to analyze the asymptotic behavior of solutions near infinity. We deal with it in Proposition \ref{thm:asymptotic}.

The paper is organized as follows. In Section \ref{sec-2}, we prove Theorem \ref{thm:main}. In Section \ref{sec-3}, we investigate the asymptotic behavior of solutions of \eqref{eq:k-Hessian=f} near infinity, which plays a crucial role in the proof of Theorem \ref{thm:main-thm}. Section \ref{sec-4} is devoted to proving Theorem \ref{thm:main-thm}.

In the rest of the paper, we denote $B_r$ as the ball in $\mathbb{R}^n$ centered at $0$ of radius $r$, and $C(m_1,\cdots,m_j)$ as some positive constant depending only on $m_1,\cdots,m_j$ that may vary from line to line. For a real $n\times n$ symmetric matrix $\xi=(\xi_{ij})$, we denote by $\lambda(\xi)$ the eigenvalue vector of $\xi$. When $\lambda(\xi)\in\overline{\Gamma}_k$, we denote
$$F(\xi)=[\sigma_k(\lambda(\xi))]^{\frac{1}{k}}\quad\text{and}\quad F_{ij}(\xi)=D_{\xi_{ij}}F(\xi).$$

\ \\

\section{Proof of Theorem \ref{thm:main}}\label{sec-2}

In this section, we will prove Theorem \ref{thm:main} by making use of some known interior estimates for solutions of Hessian equations.

\begin{proof}[Proof of  Theorem \ref{thm:main}]
Without loss of generality, we may assume $R_0^2\geq2B$. For $R>R_0$, let
$$\Omega_R=\{x\in\mathbb{R}^n: u(Rx)< R^2\},$$
and
$$v(x)=\frac{u(Rx)-R^2}{R^2}.$$
Clearly, $v$ satisfies
\begin{equation*}
\begin{cases}
\sigma_k(\lambda(D^2v))=1&\quad\mathrm{in}~\Omega_R,\\
v=0&\quad\mathrm{on}~\partial\Omega_R.
\end{cases}
\end{equation*}
By \eqref{eq:quad-growth2}, we have
\begin{equation}\label{eq:O-r-R}
B_{\sqrt{\frac{1}{A_2}-\frac{B}{A_2R^2}}}\subset\Omega_R\subset B_{\sqrt{\frac{1}{A_1}}}\subset B_{1+\sqrt{\frac{1}{A_1}}}
\end{equation}
and
$$A_1|x|^2-1\leq v(x)\leq A_2|x|^2-1+\frac{B}{R^2}\leq A_2|x|^2-\frac{1}{2}.$$
It follows that
$$\|v\|_{L^\infty(B_{1+\sqrt{\frac{1}{A_1}}})}\leq C(A_1,A_2).$$
In view of \eqref{eq:O-r-R}, we apply the interior gradient estimate in \cite[Theorem 3.2]{Chou-Wang 2001} to $v$ in $B_{1+\sqrt{\frac{1}{A_1}}}$ and obtain
$$\|Dv\|_{L^\infty(\Omega_R)}\leq C.$$
Here and in the following, $C\geq1$ denotes some constant depending only on $n$, $k$, $A_1$, $A_2$, $B$ and $R_0$ unless otherwise stated.

For $M<0$, let
$$O_M=\{x\in\Omega_R: v(x)<M\}.$$
Taking $M=-\frac{1}{4}$, we apply the second derivative estimate in \cite[Theorem 1.5]{Chou-Wang 2001} to $v$ in $O_{M}$ and obtain
$$\Big(v+\frac{1}{4}\Big)^4|D^2v|\leq C\quad\mathrm{in}\ O_{-\frac{1}{4}}.$$
This yields that
$$|D^2v|\leq C\quad\mathrm{in}\ O_{-\frac{1}{3}}.$$
Note that $v(x)\leq A_2|x|^2-\frac{1}{2}$. We thus get
$$B_{\sqrt{\frac{1}{8A_2}}}\subset O_{-\frac{1}{3}}.$$
It follows from the Evans-Krylov theorem that
$$\|D^2v\|_{C^\alpha(\overline{B}_{\sqrt{\frac{1}{8A_2}}})}\leq C,\quad\forall\ 0<\alpha<1.$$
where $C$ depends in addition on $\alpha$. Hence
$$\|D^2u\|_{C^\alpha(\overline{B}_{\frac{R}{\sqrt{8A_2}}})}=R^{-\alpha}\|D^2v\|_{C^\alpha(\overline{B}_{\sqrt{\frac{1}{8A_2}}})}\leq CR^{-\alpha}.$$
Since $R>R_0$ is arbitrary, by letting $R\rightarrow\infty$, we obtain
$$\|D^2u\|_{C^\alpha(\mathbb{R}^n)}=0.$$
Therefore, $u$ is a quadratic polynomial. This finishes the proof of Theorem \ref{thm:main}.
\end{proof}

\section{Asymptotic behavior of solutions near infinity}\label{sec-3}

In this section we analyze the behavior of solutions near infinity. We will show that any solution that satisfies a certain growth condition approaches a quadratic polynomial at a certain speed near infinity. This key result will be used to construct entire solutions in Section \ref{sec-4}. Our proof employs iterative arguments.

Throughout this section, we always denote
$A=\text{diag}(a_1,a_2,\cdots,a_n)$ and $a=(a_1,a_2,\cdots,a_n)$.

\begin{prop}\label{thm:asymptotic}
Let $n\geq 3$ and $r_0$ be a positive number. Suppose that $f\in C^3(\mathbb{R}^n\setminus B_{r_0})$ satisfies $\inf_{\mathbb{R}^n\setminus B_{r_0}}f\geq\frac{1}{c_0}$ and
\begin{equation}\label{eq:f-deri}
|D^m(f(x)-1)|\leq c_0|x|^{-\beta-m},\quad\forall\ |x|>r_0,\ m=0,1,2,3
\end{equation}
for some $\beta>2$ and $c_0>0$. Let $u\in C^{4,\gamma}_{\mathrm{loc}}(\mathbb{R}^n\setminus B_{r_0})$ be a $k$-convex solution of
\begin{equation*}
\sigma_k(\lambda(D^2u))=f\quad\text{in}\ \mathbb{R}^n\setminus\overline{B}_{r_0},
\end{equation*}
satisfying
\begin{equation}\label{eq:u-growth}
\left|u(x)-\frac{1}{2}x^TAx\right|\leq c_1|x|^{2-\varepsilon},\quad\forall\ |x|>r_0
\end{equation}
for some $0<\gamma<1$, $A\in\mathcal{A}_k$ and $\varepsilon$, $c_1>0$. Then there exist $b\in\mathbb{R}^n$ and $c\in\mathbb{R}$ such that \eqref{eq:asymptotic} holds.
\end{prop}

Before the proof we first derive the power decay of derivatives.

\begin{lemma}\label{lem:deri-est}
Under the assumptions of Proposition \ref{thm:asymptotic}, let
\begin{equation*}\label{eq:w-bdd}
w(x)=u(x)-\frac{1}{2}x^TAx~\quad\text{for}\ |x|>r_0.
\end{equation*}
Then there exist constants $C=C(n,k,A,r_0,\varepsilon,\beta,m,c_0,c_1)>0$ and $r_1=r_1(A,\varepsilon,c_1)>r_0$ such that for $m=0,1,2,3,4$,
\begin{equation}\label{eq:w-deri-est}
|D^mw(x)|\leq C|x|^{2-\min\{\varepsilon,\beta\}-m}\quad\text{for}~|x|>r_1.
\end{equation}
\end{lemma}

\begin{proof}
For $s>0$, let
$$D_s=\left\{x\in\mathbb{R}^n:\frac{1}{2}x^TAx<s\right\}.$$
For $x$ with $|x|>2r_0$ and $R=(\min_{1\leq i\leq n}a_i)^{\frac{1}{2}}|x|$, let
$$u_R(y)=\left(\frac{4}{R}\right)^2u\left(x+\frac{R}{4}y\right)\quad\text{for}~y\in D_2,$$
and
$$w_R(y)=\left(\frac{4}{R}\right)^2w\left(x+\frac{R}{4}y\right)\quad\text{for}~y\in D_2.$$
Then
\begin{equation*}
\begin{split}
w_R(y)&=u_R(y)-\frac{8}{R^2}\Big(x+\frac{R}{4}y\Big)^TA\Big(x+\frac{R}{4}y\Big)\\
&=u_R(y)-\left(\frac{1}{2}y^TAy+\frac{4}{R}x^TAy+\frac{8}{R^2}x^TAx\right)\quad\text{for}~y\in D_2.
\end{split}
\end{equation*}
Let
$$\bar{u}_R(y)=u_R(y)-\frac{4}{R}x^TAy-\frac{8}{R^2}x^TAx\quad\text{for}~y\in D_2.$$
It is clear that
$$\sigma_k(\lambda(D^2\bar{u}_R))=\bar{f}_R(y):=f\Big(x+\frac{R}{4}y\Big)\quad\text{in}\ D_2.$$
By \eqref{eq:f-deri}, we have for $m=0,1,2,3$,
$$\|\bar{f}_R-1\|_{C^{m}(\overline{D}_2)}\leq C(A,\beta,m,c_0)R^{-\beta}.$$
By \eqref{eq:u-growth}, we have for $y\in D_2$,
\begin{equation}\label{eq:v-R-bdd}
\Big|\bar{u}_R(y)-\frac{1}{2}y^TAy\Big|=|w_R(y)|\leq\frac{16c_1}{R^2}\Big|x+\frac{R}{4}y\Big|^{2-\varepsilon}\leq C(A,\varepsilon
,c_1)R^{-\varepsilon},
\end{equation}
which particularly yields
$$\|\bar{u}_R\|_{L^\infty(D_2)}\leq C(A,r_0,\varepsilon,c_1).$$

For $M\leq2$, let
$$\Omega_{M,R}=\{y\in D_2: \bar{u}_R(y)<M\}.$$
In view of \eqref{eq:v-R-bdd}, we have
$$\Omega_{1.5,R}\subset D_{1.6}$$
for $R>r_1$ with $r_1=r_1(A,\varepsilon,c_1)>r_0$ sufficiently large. Applying the interior gradient estimate in \cite[Theorem 3.2]{Chou-Wang 2001} to $\bar{u}_R$ in $D_2$, we obtain
$$\|\nabla\bar{u}_R\|_{L^{\infty}(D_{1.6})}\leq C.$$
Here and in the following, $C\geq1$ denotes some constant depending only on $n$, $k$, $A$, $r_0$, $\varepsilon$, $\beta$, $m$, $c_0$ and $c_1$ unless otherwise stated. Applying the interior second derivatives estimate in \cite[Theorem 1.5]{Chou-Wang 2001} to $\bar{u}_R$ in $\Omega_{1.5,R}$, we further obtain
$$(\bar{u}_R(y)-1.5)^4|D^2\bar{u}_R(y)|\leq C\quad\text{for}\ y\in\Omega_{1.5,R},$$
and so
$$||D^2\bar{u}_R||_{L^\infty(\Omega_{1.2,R})}\leq C.$$
It follows that
$$D^2u_R=D^2\bar{u}_R\leq CI\quad\text{in}~D_{1.1}.$$
This implies that the operator $F$ is uniformly elliptic in $D_{1.1}$. Combining its concavity, the interior derivative estimate (see, e.g., \cite[chapter 17.4]{Gilbarg-Trudinger}) yields that
$$||u_R||_{C^{4,\alpha}(\overline{D}_1)}\leq C,\quad\forall\ 0<\alpha<1,$$
where $C$ depends in addition on $\alpha$.

It is easy to see
\begin{equation}\label{eq:w_R}
||w_R||_{C^{4,\alpha}(\overline{D}_1)}\leq C\quad\text{and}\quad A+D^2w_R\leq CI\quad\text{in}\ D_1.
\end{equation}
Clearly, $w_R$ satisfies
$$a^R_{ij}(y)D_{ij}w_R(y)=F(A+D^2w_R(y))-F(A)=\bar{f}_R^{\frac{1}{k}}(y)-1\quad\text{for}\ y\in\ D_1,$$
where $$a^R_{ij}(y)=\int_0^1F_{{ij}}(A+sD^2w_R(y))\,\text{d}s.$$
By \eqref{eq:w_R}, $(a^R_{ij})$ is uniformly elliptic in $D_1$ and
$$\|a^R_{ij}\|_{C^{2,\alpha}(\overline{D}_1)}\leq C.$$
By the Schauder estimates, we have for $m=0,1,2,3,4$,
$$|D^mw_R(0)|\leq C(\|w_R\|_{L^\infty(D_1)}+\|\bar{f}_R^{\frac{1}{k}}-1\|_{C^{2,\alpha}(\overline{D}_1)})\leq CR^{-\min\{\varepsilon,\beta\}}.$$
It follows that
$$|D^mw(x)|=\left(\frac{R}{4}\right)^{2-m}|D^mw_R(0)|\leq C|x|^{2-\min\{\varepsilon,\beta\}-m}\quad\text{for}~|x|>r_1.$$
This finishes the proof.
\end{proof}

Next we give the power decay of solutions of linear elliptic equations on exterior domains. The following lemma will be used in the proof of this section several times.

\begin{lemma}\label{lem:potential}
Let $n\geq3$ and $\Omega\subset\mathbb{R}^n$ be a bounded domain containing the origin. Let $h\in C^2(\mathbb{R}^n\setminus\Omega)$ be a solution of
$$a^{ij}D_{ij}h(x)=g(x)\quad\text{for}\ x\in\mathbb{R}^n\setminus\overline{\Omega},$$
where $(a^{ij})$ is a real $n\times n$ symmetric positive definite constant matrix, and $g\in C^{0,\gamma}_{\mathrm{loc}}(\mathbb{R}^n\setminus\Omega)$ satisfies
\begin{equation}\label{eq:g_delta}
|g(x)|\leq c_2|x|^{-\delta}\quad\text{for}\ x\in\mathbb{R}^n\setminus\Omega,
\end{equation}
for some $0<\gamma<1$, $c_2>0$ and $\delta>2$. Suppose there exists a constant $h_\infty$ such that
$$h(x)\rightarrow h_\infty\quad\text{as}\ |x|\rightarrow\infty.$$
Then
\begin{equation*}
h(x)-h_\infty=
\begin{cases}
O(|x|^{2-\min\{\delta,n\}}),&\text{if}\ \delta\neq n,\\
O(|x|^{2-n}\ln |x|),&\text{if}\ \delta=n,
\end{cases}
\end{equation*}
as $|x|\rightarrow\infty$.
\end{lemma}
\begin{proof}
By an orthogonal transformation, we may assume $(a^{ij})$ is diagonal.
Let
$$y=Tx:=\left(\frac{x_1}{\sqrt{a^{11}}},\frac{x_2}{\sqrt{a^{22}}},\cdots,\frac{x_n}{\sqrt{a^{nn}}}\right)$$
and $E=T(\Omega)$. Then
$$\widetilde{h}(y):=h(\sqrt{a^{11}}y_1,\sqrt{a^{22}}y_2,\cdots,\sqrt{a^{nn}}y_n)\quad\text{for}\ y\in\mathbb{R}^n\setminus E$$
and
$$\widetilde{g}(y):=g(\sqrt{a^{11}}y_1,\sqrt{a^{22}}y_2,\cdots,\sqrt{a^{nn}}y_n)\quad\text{for}\ y\in\mathbb{R}^n\setminus E$$
satisfy
$$\Delta\widetilde{h}(y)=a^{ii}D_{ii}h(x)=g(x)=\widetilde{g}(y)\quad\text{for}\ y\in\mathbb{R}^n\setminus\overline{E}.$$
We set
$$\hat{h}(y)=\int_{\mathbb{R}^n\setminus E}\frac{1}{n(2-n)\omega_n}|z-y|^{2-n}\widetilde{g}(z)\text{d}z\quad\text{for}\ y\in\mathbb{R}^n\setminus E,$$
where $\omega_n$ is the volume of $n$-dimensional unit ball.
Then
$$\Delta\hat{h}(y)=\widetilde{g}(y)\quad\text{for}\ y\in\mathbb{R}^n\setminus\overline{E}.$$
It follows from \eqref{eq:g_delta} that
\begin{equation}\label{eq:h_delta}
|\hat{h}(y)|\leq
\begin{cases}
C|y|^{2-\delta},&\text{if}\ \delta\neq n,\\
C|y|^{2-n}\ln|y|,&\text{if}\ \delta=n.
\end{cases}
\end{equation}
Indeed, for each $y\in\mathbb{R}^n\setminus E$, let
\begin{equation*}
\begin{split}
E_1=&\left\{z\in\mathbb{R}^n\setminus E:|z|\leq\frac{|y|}{2}\right\},\\
E_2=&\left\{z\in\mathbb{R}^n\setminus E:|z-y|\leq\frac{|y|}{2}\right\},\\
E_3=&\left\{z\in\mathbb{R}^n\setminus E:|z-y|\geq|z|\right\},\\
E_4=&(\mathbb{R}^n\setminus E)\setminus(E_1\cup E_2\cup E_3).
\end{split}
\end{equation*}
Then \eqref{eq:h_delta} follows from the elementary estimates.

By the above, we conclude that
\begin{equation*}
\begin{cases}
\Delta(\widetilde{h}-\hat{h})=0&\text{in}\ \mathbb{R}^n\setminus\overline{E},\\
|\widetilde{h}(y)-h_\infty-\hat{h}(y)|\rightarrow0&\text{as}\ |y|\rightarrow\infty.
\end{cases}
\end{equation*}
Then the maximum principle implies
$$|\widetilde{h}(y)-h_\infty-\hat{h}(y)|=O(|y|^{2-n})\quad\text{as}\ |x|\rightarrow\infty.$$
The conclusion follows from \eqref{eq:h_delta}.
\end{proof}

To continue we prove a lemma that improves the estimates in Lemma \ref{lem:deri-est}. It is a key step towards Proposition \ref{thm:asymptotic}.
In the rest of this section, we write $g=f^\frac{1}{k}$.

\begin{lemma}\label{lem:deri-est-2}
Under the assumptions of Lemma \ref{lem:deri-est}, if $2\varepsilon<1$, then there exist constants $C=C(n,k,A,r_0,\varepsilon,\beta,m,c_0,c_1)>0$ and $r_2=r_2(n,k,A,r_0,\varepsilon,\beta,m,c_0,c_1)>r_1$ such that for $m=0,1,2,3,4$,
$$|D^mw(x)|\leq C|x|^{2-2\varepsilon-m}\quad\text{for}\ |x|>r_2,$$
where $r_1$ is as in Lemma \ref{lem:deri-est}.
\end{lemma}
\begin{proof}
Differentiating the equation
\begin{equation}\label{eq:F(D^2v)=g}
F(D^2u)=g\quad\text{in}\ \mathbb{R}^n\setminus\overline{B}_{r_0}
\end{equation}
with respect to $x_r$ gives
\begin{equation}\label{eq:D_ru}
a_{ij}(x)D_{ij}(D_{r}u)(x)=D_{r}g(x)\quad\text{for}\ |x|>r_0,
\end{equation}
where $a_{ij}(x)=F_{ij}(D^2u(x))$. Lemma \ref{lem:deri-est} implies for $|x|>r_1$
\begin{equation}\label{eq:a_ij-est}
|a_{ij}(x)-\bar{a}_i\delta_{ij}|\leq C|x|^{-\varepsilon}\quad\text{and}\quad|Da_{ij}(x)|\leq C|x|^{-1-\varepsilon},
\end{equation}
where
$$\bar{a}_i=F_{ii}(A)=\frac{1}{k}\sigma_{k-1;i}(a)\quad\text{and}\quad\sigma_{k-1;i}(a)=\sigma_{k-1}(a)|_{a_i=0}>0.$$
Moreover, differentiating \eqref{eq:F(D^2v)=g} with respect to $x_r$ and $x_s$ gives
$$a_{ij}D_{ij}(D_{rs}u)=D_{rs}g-D_s(a_{ij})D_{ij}(D_ru)\quad\text{in}\ \mathbb{R}^n\setminus\overline{B}_{r_1}.$$
Letting $h_1=D_{rs}u$, we rewrite the above equation as
$$\bar{a}_iD_{ii}h_1=g_1:=D_{rs}g-D_{s}(a_{ij})D_{ij}(D_ru)-(a_{ij}-\bar{a}_i\delta_{ij})D_{ij}h_1.$$
It follows from \eqref{C2}, \eqref{eq:a_ij-est} and Lemma \ref{lem:deri-est} that
\begin{equation}\label{eq:g_1-est}
|g_1(x)|\leq C|x|^{-2-2\varepsilon}\quad\text{for}\ |x|>r_1.
\end{equation}
By Lemma \ref{lem:deri-est}, we have $|D^2w(x)|\rightarrow0$ as $|x|\rightarrow\infty$, thus $h_1(x)\rightarrow a_r\delta_{rs}$ as $|x|\rightarrow\infty$.
Since $2\varepsilon<1$, Lemma \ref{lem:potential} yields that
$$|h_1(x)-a_r\delta_{rs}|\leq C|x|^{-2\varepsilon}\quad\text{for}\ |x|>r_2,$$
and so $|D^2w(x)|\leq C|x|^{-2\varepsilon}$. We thus get
$$|w|\leq C|x|^{2-2\varepsilon}\quad\text{for}\ |x|>r_2.$$
Applying Lemma \ref{lem:deri-est} to $w$, we have the result.
\end{proof}

Now we proceed with the proof of Proposition \ref{thm:asymptotic}.

\begin{proof}[Proof of Proposition \ref{thm:asymptotic}]
It suffices to prove for $\varepsilon>0$ small. Let $k_0$ be a positive integer such that $2^{k_0}\varepsilon<1$ and $2^{k_0+1}\varepsilon>1$. Let $\varepsilon_1=2^{k_0}\varepsilon$, then $1<2\varepsilon_1<2$. Applying Lemma \ref{lem:deri-est-2} $k_0$ times, we obtain for $m=0,1,2,3,4$,
\begin{equation}\label{eq:Dw-est-2}
|D^mw(x)|\leq C|x|^{2-\varepsilon_1-m}\quad\text{for}\ |x|>r_3.
\end{equation}
Here and in the following, $r_3$ denotes some sufficiently large constant depending only on $n$, $k$, $A$, $r_0$, $\varepsilon$, $\beta$, $m$, $c_0$ and $c_1$.

Let $h_1$ and $g_1$ be as in Lemma \ref{lem:deri-est-2}, i.e. $h_1=D_{rs}u$ and
$$g_1=D_{rs}g-D_{s}(a_{ij})D_{ij}(D_ru)-(a_{ij}-\bar{a}_i\delta_{ij})D_{ij}h_1.$$
In view of \eqref{eq:Dw-est-2}, corresponding to the proof in Lemma \ref{lem:deri-est-2}, we have
\begin{equation}\label{eq:a_ij-est-2}
|a_{ij}(x)-\bar{a}_i\delta_{ij}|\leq C|x|^{-\varepsilon_1}\quad\text{and}\quad|Da_{ij}(x)|\leq C|x|^{-1-\varepsilon_1}.
\end{equation}
It follows from \eqref{C2}, \eqref{eq:Dw-est-2} and \eqref{eq:a_ij-est-2} that
$$|g_1(x)|\leq C(|x|^{-2-\beta}+|x|^{-2-2\varepsilon_1})\leq C|x|^{-2-2\varepsilon_1}\quad\text{for}\ |x|>r_3.$$
Applying Lemma \ref{lem:potential} to $h_1$ and $g_1$, we get
$$|h_1(x)-a_r\delta_{rs}|=O(|x|^{-2\varepsilon_1}+|x|^{2-n})\quad\text{as}\ |x|\to\infty.$$
and so
$$|D^2w(x)|=O(|x|^{-1})\quad\text{as}\ |x|\to\infty.$$
Recalling that $D_ru$ satisfies the equation \eqref{eq:D_ru}, \cite[Theorem 4]{Gilbarg-Serrin 1955} yields that
$$D_rw(x)\rightarrow b_r\quad\text{as}\ |x|\rightarrow\infty$$
for some constant $b_r$. Denote $b=\lim_{|x|\rightarrow\infty}\nabla w(x)$.

Let
$$\widetilde{w}(x)=w(x)-b\cdot x\quad\text{for}\ |x|>r_3.$$
Then
\begin{equation}\label{eq:D_ew_1}
\tilde{a}_{ij}(x)D_{ij}\widetilde{w}=F(A+D^2\widetilde{w})-F(A)=g(x)-1\quad\text{for}\ |x|>r_3,
\end{equation}
where
$$\tilde{a}_{ij}(x)=\int_0^1F_{ij}(A+sD^2\widetilde{w}(x))\text{d}s.$$
For $e\in\mathbb{S}^{n-1}$, applying $D_e$ to $F(A+D^2\widetilde{w}(x))=g(x)$, we have
$$a_{ij}D_{ij}(D_e\widetilde{w})=D_eg\quad\text{in}\ \mathbb{R}^n\setminus\overline{B}_{r_3}.$$
Let $h_2=D_e\widetilde{w}$, we rewrite the above equation as
$$\bar{a}_iD_{ii}h_2=g_2:=D_eg-(a_{ij}-\bar{a}_i\delta_{ij})D_{ij}(D_e\widetilde{w})\quad\text{for}\ |x|>r_3.$$
Here $a_{ij}$ and $\bar{a}_i$ are as in Lemma \ref{lem:deri-est-2}.
It follows from \eqref{C2},  \eqref{eq:Dw-est-2} and \eqref{eq:a_ij-est-2} that
$$|g_2(x)|\leq C(|x|^{-\beta-1}+|x|^{-1-2\varepsilon_1})\leq C|x|^{-1-2\varepsilon_1}\quad\text{for}\ |x|>r_3.$$
Applying Lemma \ref{lem:potential} to $h_2$ and $g_2$, we get
$$|h_2(x)|\leq C(|x|^{1-2\varepsilon_1}+|x|^{2-n})\leq C|x|^{1-2\varepsilon_1}\quad\text{for}\ |x|>r_3.$$
That is
$$|\nabla\widetilde{w}(x)|\leq C|x|^{1-2\varepsilon_1}\quad\text{for}\ |x|>r_3.$$
Hence
$$|\widetilde{w}(x)|\leq C|x|^{2-2\varepsilon_1}\quad\text{for}\ |x|>r_3.$$
Applying Lemma \ref{lem:deri-est} to $\widetilde{w}$, we obtain for $m=0,1,2,3,4$,
\begin{equation}\label{eq:Dw-est-3}
|D^m\widetilde{w}(x)|\leq C|x|^{2-2\varepsilon_1-m}\quad\text{for}\ |x|>r_3.
\end{equation}
This implies
\begin{equation}\label{eq:a_ij-est-3}
|a_{ij}(x)-\bar{a}_i\delta_{ij}|\leq C|x|^{-2\varepsilon_1}\quad\text{and}\quad|Da_{ij}(x)|\leq C|x|^{-1-2\varepsilon_1}.
\end{equation}
From \eqref{C2}, \eqref{eq:Dw-est-3} and \eqref{eq:a_ij-est-3}, we have the new estimate of $g_2$
$$|g_2(x)|\leq C(|x|^{-1-\beta}+|x|^{-1-4\varepsilon_1}).$$
Since $h_2(x)\rightarrow0$ as $|x|\rightarrow\infty$, we apply Lemma \ref{lem:potential} to $h_2$ and $g_2$ and obtain the new estimate of $h_2$
\begin{equation*}
|h_2(x)|\leq
\begin{cases}
C(|x|^{1-\beta}+|x|^{1-4\varepsilon_1}+|x|^{2-n}),&\text{if}\ \min\{\beta,4\varepsilon_1\}\neq n-1,\\
C|x|^{2-n}\ln |x|,&\text{if}\ \min\{\beta,4\varepsilon_1\}=n-1.
\end{cases}
\end{equation*}
We thus have
$$|\nabla\widetilde{w}(x)|=O(|x|^{-1})\quad\text{as}\ |x|\rightarrow\infty.$$
In view of \eqref{eq:D_ew_1}, we apply \cite[Theorem 4]{Gilbarg-Serrin 1955} to $\widetilde{w}$ and obtain
$$\lim_{|x|\rightarrow\infty}\widetilde{w}(x)=c$$
for some constant $c$. We thus conclude that
$$\lim_{|x|\rightarrow\infty}\left|u(x)-\left(\frac{1}{2}x^TAx+b\cdot x+c\right)\right|=0.$$

Let
$$\overline{w}(x)=\widetilde{w}(x)-c\quad\text{for}\ |x|>r_3.$$
Then $|\overline{w}(x)|\leq C$ for $|x|>r_3$. Applying Lemma \ref{lem:deri-est} to $\overline{w}$, we obtain for $m=0,1,2,3,4$,
\begin{equation}\label{eq:Dw-est-5}
|D^m\overline{w}(x)|\leq C|x|^{-m}\quad\text{for}\ |x|>r_3.
\end{equation}
This implies
\begin{equation}\label{eq:a_ij-est-4}
|a_{ij}(x)-\bar{a}_i\delta_{ij}|\leq C|x|^{-2}.
\end{equation}
Clearly,
$$a_{ij}D_{ij}\overline{w}=F(A+D^2\overline{w})-F(A)=g-1.$$
We rewrite the above equation as
$$\bar{a}_iD_{ii}\overline{w}=g_3:=g-1-(a_{ij}-\bar{a}_i\delta_{ij})D_{ij}\overline{w}.$$
It follows from \eqref{C2}, \eqref{eq:Dw-est-5} and \eqref{eq:a_ij-est-4} that
$$|g_3(x)|\leq C(|x|^{-\beta}+|x|^{-4})\quad\text{for}\ |x|>r_3.$$
Applying Lemma \ref{lem:potential} to $\overline{w}$ and $g_3$, we obtain
\begin{equation}\label{eq:Dw-est-4}
|\overline{w}(x)|=
\begin{cases}
O(|x|^{2-\beta}+|x|^{-2}+|x|^{2-n}),&\text{if}\ \min\{\beta,4\}\neq n,\\
O(|x|^{2-n}\ln |x|),&\text{if}\ \min\{\beta,4\}=n,
\end{cases}
\end{equation}
as $|x|\to\infty$. It is easy to see that the proof is finished for $\min\{\beta,n\}\leq4$. If $\min\{\beta,n\}>4$, then
$$|\overline{w}(x)|=O(|x|^{-2})\quad\text{as}\ |x|\to\infty.$$
Applying Lemma \ref{lem:deri-est} to $\overline{w}$, we obtain for $m=0,1,2,3,4$,
\begin{equation}\label{eq:Dw-est-6}
|D^m\overline{w}(x)|\leq C|x|^{-2-m}\quad\text{for}\ |x|>r_3.
\end{equation}
This implies
\begin{equation}\label{eq:a_ij-est-5}
|a_{ij}(x)-\bar{a}_i\delta_{ij}|\leq C|x|^{-4}.
\end{equation}
From \eqref{C2}, \eqref{eq:Dw-est-6} and \eqref{eq:a_ij-est-5}, we have the new estimate of $g_3$
$$|g_3(x)|\leq C(|x|^{-\beta}+|x|^{-8})\quad\text{for}\ |x|>r_3.$$
Applying Lemma \ref{lem:potential} to $\overline{w}$ and $g_3$, we obtain
\begin{equation*}
|\overline{w}(x)|=
\begin{cases}
O(|x|^{2-\beta}+|x|^{-6}+|x|^{2-n}),&\text{if}\ \min\{\beta,8\}\neq n,\\
O(|x|^{2-n}\ln |x|),&\text{if}\ \min\{\beta,8\}=n,
\end{cases}
\end{equation*}
as $|x|\to\infty$. Apply the same argument as above finite times, we can remove the term $|x|^{-2}$ from \eqref{eq:Dw-est-4}. Eventually, we get
\begin{equation*}
|\overline{w}(x)|=
\begin{cases}
O(|x|^{2-\beta}+|x|^{2-n}),&\text{if}\ \beta\neq n,\\
O(|x|^{2-n}\ln |x|),&\text{if}\ \beta=n,
\end{cases}
\end{equation*}
as $|x|\to\infty$. This completes the proof of Proposition \ref{thm:asymptotic}.
\end{proof}

\section{Proof of Theorem \ref{thm:main-thm}}\label{sec-4}

In this section, we shall prove Theorem \ref{thm:main-thm}. We mainly follow the idea of \cite{CL 2003} and \cite{JX 2016}, and suitably modify for our situation.

Let $A=\text{diag}(a_1,\cdots,a_n)\in\mathcal{A}_k$. Following the notations used in \cite{BLL 2014}, we call $u$ a generalized symmetric function with respect to $A$ if it is a function of $s=\frac{1}{2}\sum_{i=1}^na_ix_i^2$, that is
$$u(x)=u(s):=u\left(\frac{1}{2}\sum_{i=1}^na_ix_i^2\right).$$
We denote $u'(s)=\frac{\mathrm{d}u}{\mathrm{d}s}$ and $u''(s)=\frac{\mathrm{d^2}u}{\mathrm{d}s^2}$.

Since $f$ satisfies \eqref{C1}-\eqref{C2}, there exist $C_0$, $s_0>1$ such that for $s\geq s_0$,
$$f(x)\leq\overline{f}(x)=\overline{f}(s):=1+C_0s^{-\frac{\beta}{2}},$$
and
$$f(x)\geq\underline{f}(x)=\underline{f}(s):=1-C_0s^{-\frac{\beta}{2}}>0.$$
For $s>0$, let
$$D_s=\left\{x\in\mathbb{R}^n:\frac{1}{2}x^TAx<s\right\}.$$
For the case that $f\in C^\infty(\mathbb{R}^n)$, let $u_s\in C^\infty(\overline{D}_s)$ be the unique $k$-convex solution of
\begin{equation*}
\begin{cases}
\sigma_{k}(\lambda(D^2u_s))=f&\mathrm{in}~D_s,\\
u_s=s&\mathrm{on}~\partial D_s.
\end{cases}
\end{equation*}
Here the existence of $u_s$ is guaranteed by \cite[Theorem 3]{CNS 1985}.

\begin{lemma}\label{lem:barrier}
Let $n\geq3$ and $f\in C^\infty(\mathbb{R}^n)$ satisfy \eqref{C1}-\eqref{C2}. Then for $s\geq s_0$, there exists a positive constant $C_1$ such that
\begin{equation}\label{eq:u_s-bdd0}
\sup_{D_s}\left|u_s(x)-\frac{1}{2}\sum_{i=1}^na_ix_i^2\right|<C_1,
\end{equation}
where $C_1$ depends only on $n$, $k$, $A$, $C_0$, $s_0$, $\beta$ and $\|f\|_{L^{\frac{n}{k}}(D_{s_0})}$.
\end{lemma}

\begin{proof}
We begin by constructing subsolutions and supersolutions. Let $\eta$ be a nonnegative smooth function supported in $D_{\frac{s_0}{4}}$ satisfying $\|\eta\|_{L^{\frac{n}{k}}(D_{s_0})}=1$
and $v_1\in C^\infty(\overline{D}_{s_0})$ be the $k$-convex solution of
\begin{equation*}
\begin{cases}
\sigma_k(\lambda(D^2v_1))=f+c_0\eta&\text{in}\ D_{s_0},\\
v_1=0&\text{on}\ \partial D_{s_0},
\end{cases}
\end{equation*}
where $c_0>0$ will be chosen later. Let $v_2\in C^\infty(\overline{D}_{s_0})$ be the convex solution of
\begin{equation*}
\begin{cases}
\text{det}(D^2v_2)=(f+c_0\eta)^{\frac{n}{k}}&\text{in}\ D_{s_0},\\
v_2=0&\text{on}\ \partial D_{s_0}.
\end{cases}
\end{equation*}
By the Maclaurin inequality, we have
$$(\text{det}(D^2v_2))^{\frac{1}{n}}\leq\left(\frac{1}{C_n^{k}}\sigma_{k}(\lambda(D^2v_2))\right)^\frac{1}{k},$$
where $C_n^k=\frac{n!}{k!(n-k)!}$. Thus
$$\sigma_k(\lambda(D^2v_2))\geq f+c_0\eta\quad\text{in}\ D_{s_0}.$$
By the comparison principle, we have  $v_1\geq v_2$ in $D_{s_0}$. It follows from the Alexandrov's maximum principle (see, e.g., \cite[Theorem 1.4.2]{Gutierrez 2001}) that
\begin{equation*}
\begin{split}
v_2&\geq-C(n,A,s_0)\left(\int_{D_{s_0}}(f(x)+c_0\eta(x))^{\frac{n}{k}}\text{d}x\right)^{\frac{1}{n}}\\
&\geq-C(n,A,s_0)\left(\|f\|_{L^{\frac{n}{k}}(D_{s_0})}+c_0\right)^{\frac{1}{k}}\\
&=:-c_1\quad\text{in}\ D_{\frac{s_0}{2}}.
\end{split}
\end{equation*}
Let
$$\tau=\frac{1}{2}\sum_{i=1}^na_ix_i^2,\quad c_2= \int_{\frac{s_0}{2}}^{s_0}\left(\int_{\frac{s_0}{2}}^{t}nr^{n-1}\overline{f}^{\frac{n}{k}}(r)\text{d}r\right)^{\frac{1}{n}}\text{d}t,\quad H_1=\frac{c_1}{c_2}$$
and
\begin{equation*}
v_3(x)=
\begin{cases}
-c_1,&\text{for}\ 0\leq\tau<\frac{s_0}{2},\\
\displaystyle H_1\int_{s_0}^{\tau}\left(\int_{\frac{s_0}{2}}^{t}nr^{n-1}\overline{f}^{\frac{n}{k}}(r)\text{d}r\right)^{\frac{1}{n}}\text{d}t,&\text{for}\ \frac{s_0}{2}\leq\tau\leq s_0.
\end{cases}
\end{equation*}
Firstly, $v_2\geq v_3$ in $\overline{D}_{\frac{s_0}{2}}$. Secondly, using the explicit formula \cite[Lemma 1.3]{BLL 2014} that Hessian operators act on generalized symmetric functions, direct computation shows
\begin{equation*}
\begin{split}
\text{det}(D^2v_3)&=\det(A)(v_3')^n+v_3''(v_3')^{n-1}\sum_{i=1}^n\sigma_{n-1;i}(a)(a_ix_i)^2\\
&\geq2\text{det}(A)v_3''(v_3')^{n-1}\tau\\
&\geq\text{det}(A)H_1^n\overline{f}^{\frac{n}{k}}\quad\text{in}\ D_{s_0}\setminus\overline{D}_{\frac{s_0}{2}},
\end{split}
\end{equation*}
where $a=(a_1,\cdots,a_n)$ and $\sigma_{n-1;i}(a)=\sigma_{n-1}(a)|_{a_i=0}$. By taking $c_0$ sufficiently large such that $c_1\geq c_2(\text{det}(A))^{-\frac{1}{n}}$, we have
$$\text{det}(D^2v_3)\geq\overline{f}^{\frac{n}{k}}\geq f^{\frac{n}{k}}=\text{det}(D^2v_2)\quad\text{in}\ D_{s_0}\setminus\overline{D}_{\frac{s_0}{2}},$$
and $v_2=v_3=0$ on $\partial D_{s_0}$. By the comparison principle, we have $v_2\geq v_3$ in $D_{s_0}\setminus\overline{D}_{\frac{s_0}{2}}$, and so $v_2\geq v_3$ in $D_{s_0}$.

Let
\begin{equation*}
\underline{u}(x)=
\begin{cases}
v_1(x),&\text{for}\ 0\leq\tau<s_0,\\
\displaystyle \int_{s_0}^\tau\left(t^{-\kappa}\left(\int_{s_0}^t\kappa r^{\kappa-1}\overline{f}(r)\text{d}r+H_2\right)\right)^\frac{1}{k}\text{d}t,&\text{for}\ \tau\geq s_0,
\end{cases}
\end{equation*}
where $\kappa=\frac{k}{2h_k(a)}$, $h_k(a)=\max\limits_{1\leq i\leq n}a_i\sigma_{k-1;i}(a)$
and $H_2>0$ will be chosen later.
Then $\underline{u}\in C^0(\mathbb{R}^n)\cap C^\infty(\overline{D}_{s_0})\cap C^\infty(\mathbb{R}^n\setminus D_{s_0})$,
\begin{equation}\label{eq:underline_u1}
\sigma_k(\lambda(D^2\underline{u}))\geq f\quad\text{in}\ D_{s_0}.
\end{equation}
For $\tau>s_0$, it follows that
$$\underline{u}'(\tau)=\left(\tau^{-\kappa}\left(\int_{s_0}^\tau\kappa r^{\kappa-1}\overline{f}(r)\text{d}r+H_2\right)
\right)^{\frac{1}{k}}>0$$
and
\begin{equation*}
\begin{split}
\underline{u}''(\tau)&=-\frac{\tau^{-\frac{k}{2h_k}-1}}{2h_k(a)}(\underline{u}')^{1-k}
\left(\int_{s_0}^\tau\kappa r^{\kappa-1}\overline{f}(r)\text{d}r+H_2-\tau^{\kappa}\overline{f}(\tau)\right)\\
&=:-\frac{\tau^{-\frac{k}{2h_k}-1}}{2h_k(a)}(\underline{u}')^{1-k}\mathcal{H}(\tau),
\end{split}
\end{equation*}
where
\begin{equation*}
\mathcal{H}(\tau)=
\begin{cases}\displaystyle
H_2+\frac{C_0\beta}{2\kappa-\beta}\tau^{\kappa-\frac{\beta}{2}}-\frac{2C_0\kappa}{2\kappa-\beta}s_0^{\kappa-\frac{\beta}{2}}-s_0^{\kappa},&\text{if}\ \kappa\neq\frac{\beta}{2},\\
H_2+C_0\kappa\ln\tau-C_0\kappa\ln s_0-s_0^{\kappa}-C_0,&\text{if}\ \kappa=\frac{\beta}{2}.
\end{cases}
\end{equation*}
Therefore, there exists $\widetilde{H}>0$, depending only on $C_0$, $s_0$, $\kappa$ and $\beta$, such that $\underline{u}''<0$ when $H_2\geq\widetilde{H}$. Using \cite[Lemma 1.3]{BLL 2014} again, we have
\begin{equation}\label{eq:underline_u2}
\begin{split}
\sigma_k(\lambda(D^2\underline{u}))&=(\underline{u}')^k+\underline{u}''(\underline{u}')^{k-1}\Sigma_{i=1}^n\sigma_{k-1;i}(a)(a_ix_i^2)\\
&\geq(\underline{u}')^k+2h_k(a)\underline{u}''(\underline{u}')^{k-1}\tau\\
&=\overline{f}\quad\text{in}\ \mathbb{R}^n\setminus\overline{D}_{s_0}.
\end{split}
\end{equation}
Moreover, we have
$$\underline{u}\geq v_3\quad\text{in}\ D_{s_0}\quad\text{and}\quad\underline{u}=v_3\quad\text{on}\ \partial D_{s_0}.$$
Then
$$\lim_{t\to0^+}\frac{\underline{u}(x)-\underline{u}(x-t\nu)}{t}\leq\lim_{t\to0^+}\frac{v_3(x)-v_3(x-t\nu)}{t}
\quad\text{for}\ x\in\partial D_{s_0},$$
where $\nu$ is the unit outer normal vectors of $\partial D_{s_0}$. By taking $H_2\geq\widetilde{H}$ sufficiently large, we have
$$\lim_{\tau\rightarrow s_0^-}v_3'(\tau)=H_1\left(\int_{\frac{s_0}{2}}^{s_0}nr^{n-1}\overline{f}^{\frac{n}{k}}(r)\text{d}r\right)^{\frac{1}{n}}
<(H_2s_0^{-\kappa})^{\frac{1}{k}}
=\lim_{\tau\rightarrow s_0^+}\underline{u}'(\tau).$$
Hence
\begin{equation}\label{eq:nu+-}
\lim_{t\to0^+}\frac{\underline{u}(x)-\underline{u}(x-t\nu)}{t}<\lim_{t\to0^+}\frac{\underline{u}(x+t\nu)-\underline{u}(x)}{t}
\quad\text{for}\ x\in\partial D_{s_0}.
\end{equation}
Also, by a simple computation, we have
\begin{equation}\label{eq:underline2_u-bdd1}
\sup_{\mathbb{R}^n}\left|\underline{u}(x)-\frac{1}{2}\sum_{i=1}^na_ix_i^2\right|\leq C,
\end{equation}
for some $C\geq1$ depending only on $n$, $k$, $A$, $C_0$, $s_0$, $\beta$ and $\|f\|_{L^{\frac{n}{k}}(D_{s_0})}$.

We proceed to define
\begin{equation*}
\overline{u}(x)=
\begin{cases}
0,&\text{for}\ 0\leq\tau<s_0,\\
\displaystyle \int_{s_0}^\tau\left(t^{-\kappa}\int_{s_0}^t\kappa r^{\kappa-1}\underline{f}(r)\text{d}r\right)^{\frac{1}{k}}\text{d}t,&\text{for}\ \tau\geq s_0.
\end{cases}
\end{equation*}
Then $\overline{u}\in C^0(\mathbb{R}^n)\cap C^{\infty}(\overline{D}_{s_0})\cap C^{\infty}(\mathbb{R}^n\setminus D_{s_0})$
and
\begin{equation}\label{eq:w+deri}
\lim_{\tau\rightarrow s_0^-}\overline{u}'(\tau)=\lim_{\tau\rightarrow s_0^+}\overline{u}'(\tau)=0.
\end{equation}
Arguing as above, we infer that
\begin{equation*}
\begin{split}
\sigma_k(D^2\overline{u})\leq f\quad&\text{in}~D_{s_0},\\
\sigma_k(D^2\overline{u})\leq\underline{f}\quad&\text{in}~\mathbb{R}^n\setminus\overline{D}_{s_0},
\end{split}
\end{equation*}
and
\begin{equation}\label{eq:u+bdd}
\sup_{\mathbb{R}^n}\left|\underline{u}(x)-\frac{1}{2}\sum_{i=1}^na_ix_i^2\right|\leq C.
\end{equation}
We conclude from \eqref{eq:underline2_u-bdd1} and \eqref{eq:u+bdd} that
$$\beta_{-}:=\inf_{\mathbb{R}^n}\left(\frac{1}{2}\sum_{i=1}^na_ix_i^2-\underline{u}(x)\right)>-\infty,$$
and
$$\beta_{+}:=\sup_{\mathbb{R}^n}\left(\frac{1}{2}\sum_{i=1}^na_ix_i^2-\overline{u}(x)\right)<\infty.$$

Next, we will show that for $s>s_0$,
\begin{equation}\label{eq:us-bdd}
\underline{u}(x)+\beta_{-}\leq u_s(x)\leq\overline{u}(x)+\beta_{+}\quad\text{for}~x\in D_s.
\end{equation}
To establish the first inequality, let $\underline{x}$ be a maximum point of the function
$$\underline{h}(x):=\underline{u}(x)+\beta_{-}-u_s(x)$$
in $\overline{D}_s$. It follows that \eqref{eq:underline_u1} and \eqref{eq:underline_u2} that
$$\sigma_k(\lambda(D^2\underline{u}))\geq\sigma_k(\lambda(D^2u_s))\quad\text{in}~D_{s_0},$$
and
$$\sigma_k(\lambda(D^2\underline{u}))\geq\sigma_k(\lambda(D^2u_s))\quad\text{in}~\mathbb{R}^n\setminus\overline{D}_{s_0}.$$
Then we have, by the strong maximum principle, $\underline{x}\in\partial D_{s_0}$ or $\underline{x}\in\partial D_s$. If $\underline{x}\in\partial D_s$, then in view of the boundary data of $u_s$ and the definition of $\beta_{-}$,
$$\underline{h}(x)\leq \underline{h}(\underline{x})=\underline{u}(\underline{x})+\beta_{-}-u_s(\underline{x})\leq\frac{1}{2}\sum_{i=1}^na_i\underline{x}_i^2-s=0
\quad\text{for}\ x\in\overline{D}_{s},$$
and so the inequality holds. If $\underline{x}\in\partial D_{s_0}$, then
$$\lim_{t\to0^+}\frac{h(\underline{x})-h(\underline{x}-t\nu)}{t}\geq0\geq\lim_{t\to0^+}\frac{h(\underline{x}+t\nu)-h(\underline{x})}{t},$$
which contradicts to \eqref{eq:nu+-} by recalling the smoothness of $u_s$.
Hence, the first inequality in \eqref{eq:us-bdd} holds. For the second inequality, let $\overline{x}$ be a minimum point of the function
$$\overline{h}(x):=\overline{u}(x)+\beta_{+}-u_s(x)$$
in $\overline{D}_s$. Again the strong maximum principle yields that $\overline{x}\in\partial D_{s_0}$ or $\overline{x}\in\partial D_s$. If $\overline{x}\in\partial D_s$, then by the definition of $\beta_{+}$, we have
$$\overline{h}(x)\geq\overline{h}(\overline{x})=\overline{u}(\overline{x})+\beta_{+}-u_s(\overline{x})
\geq\frac{1}{2}\sum_{i=1}^na_i\overline{x}_i^2-s=0\quad\text{in}~\overline{D}_s,$$
and so the second inequality holds. If $\bar{x}\in\partial D_{s_0}$, in view of \eqref{eq:w+deri} and the smoothness of $u_s$, we have
$$\nabla u_s(\bar{x})=0,$$
which is impossible due to Hopf's Lemma. Hence, the second inequality in of \eqref{eq:us-bdd} holds.
This completes the proof.
\end{proof}

Note that Lemma \ref{lem:barrier} gives an estimate that depends only on $L^{\frac{n}{k}}$ norm of $f$ instead of some stronger norm of $f$.

\begin{proof}[Proof of Theorem \ref{thm:main-thm}]
We only show the existence part as the uniqueness part follows immediately from the comparison principle. For the existence part, by an orthogonal transformation and subtracting a linear function, we only need to prove for the case that $A=\text{diag}(a_1,\cdots,a_n)$, $b=0$ and $c=0$.
Without loss of generality, we may assume $D\subset D_{\frac{s_0}{2}}$.

First, we explore the proof under the hypothesis that $f\in C^\infty(\mathbb{R}^n)$. We will show that along a sequence $s\rightarrow\infty$, $u_s$ converges to a solution $u$ of \eqref{eq:k-Hessian=f}
satisfying
\begin{equation}\label{eq:u_s-bdd}
\sup_{\mathbb{R}^n}\left|u(x)-\frac{1}{2}\sum_{i=1}^na_ix_i^2\right|\leq C_1,
\end{equation}
where $C_1$ is as in Lemma \ref{lem:barrier}. For this purpose, we are going to derive the locally uniform estimates of $u_s$.

For any fixed compact subset $K$ of $\mathbb{R}^n$, let
$$K_1=\{x\in\mathbb{R}^n:\text{dist}(x,K)\leq1\},$$
$$M=\sup_{K_1}\left(\frac{1}{2}\sum_{i=1}^na_ix_i^2\right)+C_1+1,$$
and
$$\Omega_{s,M}=\{x\in D_s:u_s(x)<M\}.$$
It is easy to check that
\begin{equation}\label{eq:O-R-bdd}
K_1\subset\Omega_{s,M-1}\subset\Omega_{s,M}\subset D_{s_1}\quad\text{for}\ s>s_1
\end{equation}
with $s_1>s_0$ sufficiently large. Indeed, taking $s_1\geq M+C_1$, it follows from Lemma \ref{lem:barrier} that
$$u_s(x)<\frac{1}{2}\sum_{i=1}^na_ix_i^2+C_1\leq M-1\quad\text{in}~K_1,$$
and
$$u_s(x)>\frac{1}{2}\sum_{i=1}^na_ix_i^2-C_1\geq s_1-C_1\geq M\quad\text{in}~D_s\setminus D_{s_1}.$$
Moreover, Lemma \ref{lem:barrier} yields that
$$||u_s||_{L^\infty(D_{s_1+1})}\leq C_1+s_1+1\quad\text{for}\ s\geq s_1+1.$$
By \eqref{eq:O-R-bdd}, we apply the interior gradient estimate in \cite[Theorem 3.2]{Chou-Wang 2001} to $u_s$ in $D_{s_1+1}$  and obtain
$$||\nabla u_s||_{L^{\infty}(\Omega_{s,M})}\leq C(n,k,A,C_1,K,\|f\|_{C^{0,1}(D_{s_1+1})})\quad\text{for}\ s\geq s_1+1.$$
The second derivative estimate in \cite[Theorem 1.5]{Chou-Wang 2001} further yields that
$$(u_s(x)-M)^4|D^2u_s(x)|\leq C(n,k,A,C_1,K,\|f\|_{C^{1,1}(D_{s_1+1})})\quad\text{in}~\Omega_{s,M}.$$
It follows that
$$|D^2u_s|\leq C\quad\text{in}~\Omega_{s,M-1}.$$
Hence, the operator $F$ is uniformly elliptic with respect to $u_s$ in $\Omega_{s,M-1}$. Combining its concavity,
the Evans-Krylov estimates and Schauder theory imply that for $m\geq4$ and $0<\alpha<1$,
$$||u_s||_{C^{m,\alpha}(K)}\leq C(n,k,A,m,\alpha,C_1,K,||f||_{C^{m-2,\alpha}(D_{s_1+1})}).$$
Up to a subsequence $s_i\rightarrow\infty$, we get
$$u_{s_i}\rightarrow u_\infty\quad\text{in}~C^m_{\text{loc}}(\mathbb{R}^n),\quad\forall\ m\geq4.$$
This particularly implies that $u_\infty\in C^\infty(\mathbb{R}^n)$ is a $k$-convex solution of
$$\sigma_k(\lambda(D^2u_\infty))=f\quad\text{in}~\mathbb{R}^n,$$
satisfying
$$\sup_{\mathbb{R}^n}\left|u_\infty(x)-\frac{1}{2}\sum_{i=1}^na_ix_i^2\right|\leq C_1$$
and for $m\geq4$ and $0<\alpha<1$,
$$\|u_\infty\|_{C^{m,\alpha}(K)}\leq C(n,k,A,m,\alpha,C_1,K,||f||_{C^{m-2,\alpha}(D_{s_1+1})}).$$

For general $f\in C^{2,\gamma}_{\text{loc}}(\mathbb{R}^n)$, let $f^\varepsilon=\rho_\varepsilon*f$, where $\rho_\varepsilon$ is the standard mollifier. Let $u^{\varepsilon}_\infty$ be the solution found above for $f^\varepsilon$. From the proof, we see that
$$\sup_{\mathbb{R}^n}\left|u^\varepsilon_\infty(x)-\frac{1}{2}\sum_{i=1}^na_ix_i^2\right|\leq C_1$$
and
$$\|u^\varepsilon_\infty\|_{C^{4,\gamma}(K)}\leq C(n,k,A,\gamma,C_1,K,||f^\varepsilon||_{C^{2,\gamma}(D_{s_1+1})}),\quad\forall\ K\Subset\mathbb{R}^n.$$
Since $f\in C^{2,\gamma}_{\text{loc}}(\mathbb{R}^n)$, we have, up to a subsequence $\varepsilon_i\to0$,
$$u^{\varepsilon_i}_\infty\rightarrow u^0_\infty\quad\text{in}\ C^4_{\text{loc}}(\mathbb{R}^n)$$
for some $k$-convex function $u^0_\infty$. Hence $u_\infty^0\in C^{4,\gamma}_{\text{loc}}(\mathbb{R}^n)$ is a solution of \eqref{eq:k-Hessian=f} and satisfies
\begin{equation}\label{eq:u^0-bdd}
\sup_{\mathbb{R}^n}\left|u^0_\infty(x)-\frac{1}{2}\sum_{i=1}^na_ix_i^2\right|\leq C_1.
\end{equation}
It follows from Proposition \ref{thm:asymptotic} and \eqref{eq:u^0-bdd} that there exists $\tilde{c}\in\mathbb{R}$ such that
\begin{equation*}
\begin{cases}
\limsup_{|x|\rightarrow\infty}|x|^{\min\{\beta,n\}-2}\left|u_{\infty}^0(x)-\left(\frac{1}{2}\sum_{i=1}^na_ix_i^2+\tilde{c}\right)\right|<\infty,&\text{if}\ \beta\neq n.\\
\limsup_{|x|\rightarrow\infty}|x|^{n-2}(\ln|x|)^{-1}\left|u_{\infty}^0(x)-\left(\frac{1}{2}\sum_{i=1}^na_ix_i^2+\tilde{c}\right)\right|<\infty,&\text{if}\ \beta=n.
\end{cases}
\end{equation*}
Then
$$u:=u^0_\infty-\tilde{c}\quad\text{in}\ \mathbb{R}^n$$
is the desired solution. This completes the proof of Theorem \ref{thm:main-thm}.
\end{proof}

\end{document}